\documentclass{article}
\usepackage[paperwidth=210mm,paperheight=297mm,textheight=622pt,textwidth=468pt]{geometry}
\usepackage[shortlabels]{enumitem}
\setenumerate[1]{itemsep=0pt,partopsep=0pt,parsep=\parskip,topsep=0pt}
\usepackage[numbers,sort&compress]{natbib}

\usepackage[tworuled,linesnumbered,noline,noend]{algorithm2e}
\usepackage{amsmath,amsthm,tikz,float,amssymb}
\usepackage[colorlinks=true,breaklinks=true,bookmarks=true,urlcolor=blue,citecolor=blue,linkcolor=blue,bookmarksopen=false,draft=false]{hyperref}

\newtheorem{theorem}{Theorem}[section]
\newtheorem{lemma}[theorem]{Lemma}

\newtheorem{corollary}[theorem]{Corollary}

\newtheorem{question}{Question}[section]
\newtheorem{conjecture}{Conjecture}[section]

\usepackage{caption}
\newcommand{\fold}[1]{[#1]}
\newcommand{\foldzero}{[\overrightarrow{0}]}

\begin{document}

\title{A bridge between the minimal doubly resolving set problem in (folded) hypercubes and the coin weighing problem\thanks{Supported in part by National Natural Science Foundation of China (No. 11871222), Science and Technology Commission of Shanghai Municipality (Nos. 18dz2271000, 19JC1420100) and the Open Research Fund of Key Laboratory of Advanced Theory and Application in Statistics and Data Science-MOE, ECNU.}}
\author{Changhong Lu\ \ \ \  Qingjie Ye\thanks{Corresponding author}\\
School of Mathematical Sciences,\\
Shanghai Key Laboratory of PMMP,\\
East China Normal University,\\
Shanghai 200241, P. R. China\\
\\
Email: chlu@math.ecnu.edu.cn\\
Email: mathqjye@stu.ecnu.edu.cn}
\date{}
\maketitle

\begin{abstract}
    In this paper, we consider the minimal doubly resolving set problem in Hamming graphs, hypercubes and folded hypercubes. We prove that the minimal doubly resolving set problem in hypercubes is equivalent to the coin weighing problem. Then we answer an open question on the minimal doubly resolving set problem in hypercubes. We disprove a conjecture on the metric dimension problem in folded hypercubes and give some asymptotic results for the metric dimension and the minimal doubly resolving set problems in Hamming graphs and folded hypercubes by establishing connections between these problems. Using the Lindstr\"{o}m's method for the coin weighing problem, we give an efficient algorithm for the minimal doubly resolving set problem in hypercubes and report some new upper bounds. We also prove that the minimal doubly resolving set problem is NP-hard even restrict on split graphs, bipartite graphs and co-bipartite graphs.
\end{abstract}


\section{Introduction}
Let $G$ be a finite, connected, simple and undirected graph with vertex set $V=V(G)$ and edge set $E=E(G)$. The distance between vertices $u$ and $v$ is denoted by $d_G(u,v)$. The {\sl Cartesian product} of graphs $G$ and $H$, denoted by $G\square H$, where $V(G\square H)=\{(g,h):g\in V(G),h\in V(H)\}$, and $(g_1,h_1)(g_2,h_2)\in E(G\square H)$ if and only if $g_1=g_2$, $h_1h_2\in E(H)$ or $g_1g_2\in E(G)$, $h_1=h_2$. 
The cartesian product is associative and $G_1 \square G_2 \square \cdots \square G_d$ is well-defined.

The metric dimension problem was independently defined by \citet{slater1975leaves}, \citet{harary1976metric}. A vertex subset $S$ resolves a graph $G$ if every vertex is uniquely determined by its vector of distances to the vertices in $S$. More formally, a vertex $x$ of $G$ {\sl resolves} two vertices $u$ and $v$ of $G$ if $d_G(u,x)\ne d_G(v,x)$. A vertex subset $S$ is a {\sl resolving set} of $G$ if every two vertices in $G$ are resolved by some vertex of $S$. A resolving set $S$ of $G$ with the minimum cardinality is a {\sl metric basis} of $G$, and the size of $S$ is the {\sl metric dimension} of $G$, denoted by $\beta(G)$.

Doubly resolving sets were introduced by \citet{caceres2007metric} as a tool for researching resolving sets of Cartesian products of graph. Let $G$ be a graph of order $n\ge 2$. We say that $\{x,y\}$ {\sl doubly resolves} $\{u,v\}$, if $d_G(u,x)-d_G(u,y)\neq d_G(v,x)-d_G(v,y)$. A vertex subset $S$ of $G$ is a {\sl doubly resolving set} of $G$ if every pair of distinct vertices in $G$ is doubly resolved by some pair of vertices in $S$.  Let $\Psi(G)$ denote the minimum cardinality of a doubly resolving set of a graph $G \neq K_1$. The {\sl minimal doubly resolving set problem} is determining the minimum cardinality of a doubly resolving set for an input graph $G$.

Let $d_G(u,S)=(d_G(u,x_1),\dots,d_G(u,x_m))$ and $\overrightarrow{c}=(c,\dots,c)$, where $S=\{x_1,\dots,x_m\}$ is a subset of $V(G)$ and $c$ is a constant. Note that the dimension of $\overrightarrow{c}$ would be clear from the context. Then for every distinct vertices $u,v\in V(G)$, $S$ is a resolving set if and only if $d(u,S)-d(v,S)\neq \overrightarrow{0}$ while $S$ is a doubly resolving set if and only if $d(u,S)-d(v,S)\neq \overrightarrow{c}$ for all constant $c$. Hence $\beta(G)\le \Psi(G)$. However, there is not a function $f$ such that $\Psi(G)\le f(\beta(G))$ for all graphs $G$. In fact, \citet{caceres2007metric} proved that there is a $k$-connected graph $G_{n,k}$ such that $\beta(G_{n,k}) \le 2k$ and $\Psi(G_{n,k}) \ge 2n$ for all $k \ge 1$ and $n \ge 2$. The connection between the two problems is the following theorem that was proved in \citep{caceres2007metric}.

\begin{theorem}[\citet{caceres2007metric}]\label{betaPsi}
For all graphs $G$ and $H\neq K_1$,
\begin{equation*}
\max\{\beta(G),\beta(H)\}\le \beta(G\square H)\le \beta(G)+\Psi(H)-1.
\end{equation*}
\end{theorem}

The metric dimension arises in many diverse areas, including network discovery and verification \citep{beerliova2005network}, the robot navigation \citep{khuller1996landmarks} and chemistry \citep{chartrand2000resolvability}. Finding the doubly resolving set in graphs is equivalent to locating the source of a diffusion in complex networks \citep{chen2014approximability}. The metric dimension problem and minimal doubly resolving set problem have many interesting theoretical properties which are out of the scope of this paper. The interested reader is referred, e.g. to \citep{bailey2011base,jernando2010extremal,kratica2012minimalB}.

As far as general graphs are concerned, both problems are NP-hard. The proof for the metric dimension problem is given in \citep{khuller1996landmarks} and for the minimal doubly resolving set problem is given in \citep{kratica2009computing}. \citet{epstein2015weighted} proved that the metric dimension problem is NP-hard even for split graphs, bipartite graphs and co-bipartite graphs. Therefore, some researchers try to design heuristic algorithms to solve the problems. It has been designed the genetic algorithm (GA) to solve the metric dimension problem in \citep{kratica2009computingB} and the minimal doubly resolving set problem in \citep{kratica2009computing}. \citet{mladenovic2012variable} designed the variable neighborhood search algorithm (VNS) to solve the metric dimension problem and the minimal doubly resolving set problem. \citet{chartrand2000resolvability} and \citet{kratica2009computing} gave the 0--1 integer linear programming formulations for the metric dimension problem and the minimal doubly resolving set problem respectively.

The {\sl Hamming graph} $H_{n,q}$ is the Cartesian product of $n$ copies of the complete graph $K_q$ with $q$ vertices
\begin{equation*}
H_{n,q}=\underbrace{K_q \square K_q \square \cdots \square K_q}_n.
\end{equation*}
Specifically, the vertex of $H_{n,q}$ is an $n$-dimensional vector $u=(u_1,\dots,u_n)\in \{0,1,\dots,q-1\}^n$ and two vertices are adjacent if they differ in exactly one coordinate (see Figure \ref{h23}). The operation of addition (subtraction) in $V(H_{n,q})$ is defined by the modulo-$q$ addition (subtraction) of the corresponding vector. For example, if $x = (0,0,1,1,2,2)$ and $y=(1,2,2,1,0,2)$ are two vertices of $H_{6,3}$, then $x+y=(1,2,0,2,2,1)$ and $x-y=(2,1,2,0,2,0)$. By the definition of $H_{n,q}$, it is easy to show that $d_{H_{n,q}}(u,v)=d_{H_{n,q}}(u-v,\overrightarrow{0})=\sum_{i=1}^n 1_{u_i\neq v_i}$, where $1_{u_i\neq v_i}=1$ if $u_i\neq v_i$ and $1_{u_i\neq v_i}=0$ if $u_i= v_i$.

The $n$-dimensional {\sl hypercube} $Q_n$, also called $n$-cube, is a Cartesian product of $n$ copies of $K_2$ (see Figure \ref{q3}). Note that $Q_n=H_{n,2}$. For each $u\in V(Q_n)$, we use $\overline{u}$ to denote its opposite vertex, that is $\overline{u}=u+\overrightarrow{1}$. It is clear that $d_{Q_n}(u,v)=\sum_{i=1}^n |u_i-v_i|$ and thus $d_{Q_n}(u,\overline{v})=n-d_{Q_n}(u,v)$.


The metric dimension of the Hamming graph is connected to  Mastermind, which is a deductive game for two players, the code setter and the code breaker. The code setter chooses a secret vector
$s=(s_1,\dots,s_n)\in \{0,1,\dots,q-1\}^n$. The task of the code breaker is to infer the secret
vector by a series of questions, each a vector $t=(t_1,\dots,t_n)\in \{0,1,\dots,q-1\}^n$. The code setter answers with two integers, denoted by $a(s,t)=|\{i:s_i=t_i,1\le i\le n\}|$ and $b(s,t)=\max\{a(\tilde{s},t):\tilde{s}$ is a permutation of $s\}$. The original commercial version of the game is $n = 4$ and $q = 6$, which was invented by Mordecai Meirowitz. 
 \citet{knuth1976computer} showed that four questions suffice to determine $s$ in this case. Let $g(n,q)$ be the smallest number such that the code breaker can determine any $s$ by asking $g(n,q)$ questions at once (without waiting for the answers). \citet{chvatal1983mastermind} proved that $g(n,q)\le (4+2\log_q 2+o(1))n/ \log_q n$. \citet{kabatianski2000mastermind} showed that $\beta(H_{n,q})-(q-1)\le g(n,q)\le \beta(H_{n,q})$. Let $f(n,q)$ be the smallest number such that the code breaker can determine any $s$ by asking $f(n,q)$ questions at once without $b(s,t)$ in the answers. \citet{caceres2007metric} showed that $g(n,q)\le f(n,q)=\beta(H_{n,q})$.

It has been showed that $\beta(H_{2,q})=\Psi(H_{2,q})=\lfloor (4q-2)/3 \rfloor$ for all $q\ge 5$ by \citet{caceres2007metric} and \citet{kratica2012minimalA}. Recently, \citet{jiang2019metric} gave the following nice theorem.
\begin{theorem}[\citet{jiang2019metric}]\label{jiang2019}
$\beta(H_{n,q})=(2+o(1))n/\log_q n$ for all $q\ge 2$.
\end{theorem}

We remark that \citet{kabatiansky2018metric} proved the above theorem for $q=3,4$. For $q=2$, the metric dimension problem in hypercubes is related to the following coin weighing problem.

Given $n$ coins, some of them may be defective. We know the weight $g$ of the good coins in advance and also the weight $h \neq g$ of the defective coins. If we weigh a subset of coins with a spring scale, then the outcome will tell us precisely the number of defective coins among them. The coin weighing problem is determining the minimum number $M(n)$ of weighings by means of which the good and defective coins can be separated  under the assumption that all the family of tested subsets has to be given in advance.

More formally, the binary vector $u=(u_1,\dots,u_n)\in \{0,1\}^n$ is corresponding to a distribution of defective coins, where $u_j=1$ if and only if the $j$-th coin is defective. Similarly, the binary vector $x=(x_1,\dots,x_n)\in \{0,1\}^n$ is corresponding to a weighing, where $x_j=1$ if and only if the $j$-th coin is chosen to weigh. The outcome of a weighing is a scalar product of $x$ and $u$, that is $u\cdot x=\sum_{i=1}^n u_ix_i$. A set of binary vectors $S$ is called a {\sl weighing strategy} if for every pair of distinct vectors $u,v$, there exists $x\in S$ such that $u\cdot x\neq v\cdot x$.
The coin weighing problem was proposed for $n=5$ by \citet{shapiro1960problem} and solved by \citet{fine1960solution}. \citet{erdos1963two} presented a lower bound and \citet{lindstrom1964combinatory} (independently by \citet{cantor1966determination}) presented an upper bound. The lower bound and the upper bound are asymptotically equivalent. Almost all exact values of $M(n)$ are not known yet.

\begin{theorem}[\citet{erdos1963two,lindstrom1964combinatory,cantor1966determination}]\label{erdos-lindstrom-cantor}
\begin{equation*}
M(n)=(2+o(1))n/\log_2 n.
\end{equation*}
\end{theorem}

A surprising connection between the metric dimension problem in hypercubes and the coin weighing problem was given in \citep{sebho2004metric}.
\begin{theorem}[\citet{sebho2004metric}]\label{sebh02004}
$|\beta(Q_n)-M(n)|\le 1$.
\end{theorem}

Theorems \ref{erdos-lindstrom-cantor} and \ref{sebh02004} imply that $\beta(Q_n)=(2+o(1))n/\log_2 n$. Researchers try to get optimal upper bounds of $\beta(Q_n)$ and  $\Psi(Q_n)$ by heuristic algorithms. Besides the genetic algorithm and the variable neighborhood search algorithm as mentioned previously, some specially algorithms are designed  for the metric dimension and minimal doubly resolving set problem in hypercubes. For example, \citet{nikolic2017symmetry} designed a greedy algorithm and a dynamic programming procedure for the metric dimension of a hypercube by using the symmetry property of resolving sets to reduce the size of the feasible solution set. \citet{hertz2020ipbased} designed an IP-based swapping algorithm for the metric dimension and minimal doubly resolving set problem in hypercubes.

The {\sl folded hypercube} is a graph obtained by merging opposite vertices of a hypercube. A vertex of a {\sl folded $n$-cube} $F_n$ is denoted by $\fold{u}=\{u,\overline{u}\}$, where $u$ is a vertex of $Q_n$. $\fold{u}\fold{v}\in E(F_n)$ if and only if $uv\in E(Q_n)$ or $u\overline{v}\in E(Q_n)$ (see Figure \ref{f3}). Note that $uv\in E(Q_n)$ if and only if $\bar{u}\bar{v}\in E(Q_n)$. It is  easy to show that $d_{F_n}(\fold{u},\fold{v})=\min\{d_{Q_n}(u,v),n-d_{Q_n}(u,v)\}$. Recently, \citet{zhang2020metric} gave the upper bound of metric dimension of folded $n$-cube.
\begin{theorem}[\citet{zhang2020metric}]\label{zhang2020}
$\beta(F_n)\le n-1$ for all odd $n\ge 5$ and $\beta(F_n)\le 2n-4$ for all even $n\ge 6$.
\end{theorem}
  They raised the following conjecture.

 \begin{conjecture}[\citet{zhang2020metric}]\label{zhang2020conj}
  If $n\ge 5$ is odd, then $\beta(F_n)=n-1$.
 \end{conjecture}

\begin{figure}[ht]
\begin{minipage}[t]{0.29\textwidth}
\centering
\begin{tikzpicture}

\foreach \i in {0,1,2}
{
\foreach \j in {0,1,2}
{
\coordinate[label=180-90*\i:{$\i\j$}] (a\i\j) at (0+1.3*\i,0+1.3*\j);
\fill (a\i\j) circle[radius=2pt];
}
\draw (a\i0) -- (a\i1) -- (a\i2);
\draw (a\i2) arc (90:270:0.5 and 1.3);
}
\foreach \j in {0,1,2}
{
\draw (a0\j) -- (a1\j) -- (a2\j);
\draw (a2\j) arc (0:180:1.3 and 0.5);
}

\end{tikzpicture}
\caption{The graph $H_{2,3}$}\label{h23}
\end{minipage}
\begin{minipage}[t]{0.29\textwidth}
\centering
\begin{tikzpicture}

\foreach \i in {0,1}
{
\foreach \j in {0,1}
{
\coordinate[label=-135:{$\i0\j$}] (a\i\j) at (0+2*\i,0+2*\j);
\fill (a\i\j) circle[radius=2pt];
}
}
\draw (a00) -- (a01) -- (a11) -- (a10) -- cycle;
\foreach \i in {0,1}
{
\foreach \j in {0,1}
{
\coordinate[label=45:{$\i1\j$}] (b\i\j) at (0.707+2*\i,0.707+2*\j);
\fill (b\i\j) circle[radius=2pt];
\draw (a\i\j) -- (b\i\j);
}
}
\draw (b00) -- (b01) -- (b11) -- (b10) -- cycle;

\end{tikzpicture}
\caption{The graph $Q_3$}\label{q3}
\end{minipage}
\begin{minipage}[t]{0.4\textwidth}
\centering
\begin{tikzpicture}
\coordinate[label=-135:{$(000,111)$}] (a00) at (0,0);
\coordinate[label=135:{$(010,101)$}] (a01) at (0,2);
\coordinate[label=-45:{$(100,011)$}] (a10) at (2,0);
\coordinate[label=45:{$(110,001)$}] (a11) at (2,2);
\foreach \i in {0,1}
{
\foreach \j in {0,1}
{

\fill (a\i\j) circle[radius=2pt];
}
}

\draw (a00) -- (a01) -- (a11) -- (a10) -- cycle;
\draw (a01) -- (a10);
\draw (a00) -- (a11);
\end{tikzpicture}
\caption{The graph $F_3$}\label{f3}
\end{minipage}
\end{figure}

The paper is organized as follows. Based on  a new concept of  doubly distance resolving sets, we reveal the relationship between resolving sets and doubly resolving sets in Section \ref{ddrs}.  Using the preliminary results in  Section \ref{ddrs}, we  show that $\Psi(H_{n,q})=(2+o(1))n/\log_q n$ and  $\Psi(Q_n)=M(n)+1$ in Section \ref{cube}. Hence, we construct a bridge between  the minimal doubly resolving set problem in hypercubes and the coin weighing problem. Using the result of the coin weighing problem, we prove that $\Psi(Q_n)\le \Psi(Q_{n+1})\le \Psi(Q_n)+1$, which answer an open question in \citep{hertz2020ipbased}.
In Section \ref{folded}, by exploring the connection of the metric dimension problems between  hypercubes and folded hypercubes, we give a shorter proof for Theorem \ref{zhang2020} and disprove Conjecture  \ref{zhang2020conj}. Some asymptotic results of  $\beta(F_n)$ and $\Psi(F_n)$ are also given. In Section \ref{npc}, we prove that the minimal doubly resolving set problem is NP-hard for split graphs, bipartite graphs and co-bipartite graphs. In Section \ref{cal}, using the bridge between the minimal doubly resolving set problem in hypercubes and the coin weighing problem, we explore algorithms and experimental results for the minimal doubly resolving set problem in hypercubes and get some better upper bounds of $\Psi(Q_n)$ when $n\le 93$.  We also give precise values of $\beta(F_n)$ and $\Psi(F_n)$ when $n\le 9$.

\section{Preliminary results} \label{ddrs}
The following lemma is obvious but helpful to identify a doubly resolving set of a graph.
\begin{lemma}[\citet{kratica2009computing}]\label{small}
Let $S=\{x_1,x_2,\dots,x_m\}$ be a doubly resolving set of $G$. Then for every pair of distinct vertices $u,v\in V(G)$, there exists $x_j\in S$, such that $d_G(u,x_1)-d_G(u,x_j)\neq d_G(v,x_1)-d_G(v,x_j)$.
\end{lemma}

Now we introduce a new concept to reveal the relationship between resolving sets and doubly resolving sets of graphs. Let $G$ be a graph of order $n\ge 2$. Given a vertex $x\in V(G)$, a vertex subset $S$ of $G$ is a {\sl doubly distance resolving set} of $G$ on $x$ if every pair of vertices $\{u,v\}$ with $d_G(u,x)\neq d_G(v,x)$ is doubly resolved by some pair of vertices in $S\cup \{x\}$. In other words, $S=\{x_1,x_2,\dots,x_m\}$ is a doubly distance resolving set of $G$ on $x$ if and only if $d_G(u,x)$ is uniquely determined by the vector $(d_G(u,x)-d_G(u,x_1),d_G(u,x)-d_G(u,x_2),\dots,d_G(u,x)-d_G(u,x_m))$ for any $u\in V(G)$.

\begin{lemma}\label{lpsile}
Let $S$ be a resolving set of $G$. Let $x\in S$ and $T$ be a doubly distance resolving set of $G$ on $x$. Then $S\cup T$ is a doubly resolving set of $G$.
\end{lemma}
\begin{proof}
    Let $u,v$ be two distinct vertices of $G$. If $d_G(u,x)\neq d_G(v,x)$, then $\{u,v\}$ can be doubly resolved by the definition of $T$. If $d_G(u,x)= d_G(v,x)$, then there is a vertex $y$ such that $d_G(u,y)\neq d_G(v,y)$ by the definition of $S$ and we have $d_G(u,x)-d_G(u,y)\neq d_G(v,x)-d_G(v,y)$. It leads that $\{u,v\}$ is doubly resolved by $\{x,y\}$. Therefore, $S\cup T$ is a doubly resolving set of $G$.
\end{proof}

Let $\phi(G,x)$ denote the minimum cardinality of a doubly distance resolving set of $G$ on $x$ and $\phi(G)=\max\{\phi(G,x):x\in V(G)\}$.
\begin{theorem}\label{psile}
    Let $G$ be a graph of order $n\ge 2$. Then
    \begin{equation*}
    \phi(G)\le \Psi(G)\le \beta(G)+\phi(G).
    \end{equation*}
\end{theorem}
\begin{proof}
    For every vertex $x$, since every doubly resolving set is a doubly distance resolving set on $x$ by definition, we have $\phi(G,x)\le \Psi(G)$ and thus $\phi(G)\le \Psi(G)$.

    Let $S$ be a resolving set of $G$ with $|S|=\beta(G)$. Let $x\in S$ and $T$ be a doubly distance resolving set of $G$ on $x$ with  $|T|=\phi(G,x)$. By Lemma \ref{lpsile}, $S\cup T$ is a doubly resolving set of $G$ and then $\Psi(G)\le |S\cup T|\le |S|+|T|=\beta(G)+\phi(G,x)\le \beta(G)+\phi(G)$.
\end{proof}

A function $f:V(G)\rightarrow V(G)$ is an {\sl automorphism} of $G$ if $f$ is bijective such that $uv\in E(G)$ if and only if $f(u)f(v)\in E(G)$. A graph $G$ is called a {\sl vertex-transitive graph} if for every pair of vertices $\{x,y\}$, there is an automorphism $f$ such that $f(x)=y$.

\begin{lemma}\label{sym}
Let $G$ be a vertex-transitive graph of order $n\ge 2$. Then the following holds:
\begin{enumerate}[(a)]
\item $\phi(G)=\phi(G,x)$ for every $x\in G$. \label{a}
\item For every $x\in G$, there is a minimum (doubly) resolving set $S$ of $G$ such that $x\in S$.\label{b}
\end{enumerate}

\end{lemma}
\begin{proof}
We just prove \ref{a}. The proof of \ref{b} is similar. Let $y$ be the vertex such that $\phi(G,y)=\phi(G)$. Let $S$ be a minimum doubly distance resolving set of $G$ on $x$. By the definition of vertex-transitive graph, there is an automorphism $f:V(G)\rightarrow V(G)$ such that $f(x)=y$. Then $f(S)=\{f(s):s\in S\}$ is a doubly distance resolving set of $G$ on $y$. It leads that $\phi(G)=\phi(G,y)\le |f(S)|=\phi(G,x)\le \phi(G)$, i.e. $\phi(G)=\phi(G,x)$.
\end{proof}

\section{Hamming graphs and hypercubes}\label{cube}
Let $x,y$ be two vertices of $H_{n,q}$. It is clear that $f(u)=u-x+y$ is an automorphism of $H_{n,q}$ with $f(x)=y$. Then $H_{n,q}$ is a vertex-transitive graph. By the Lemma \ref{sym}\ref{b}, we have the following corollary, which was proved for $q=2$ on the metric dimension problem in \citep{nikolic2017symmetry}.
\begin{corollary}\label{s0}\label{e0}
There is a minimum (doubly) resolving set $S$ of $H_{n,q}$ such that $\overrightarrow{0}\in S$.\qed
\end{corollary}

\begin{lemma}
For every positive integer $n$, $\phi(H_{n,q})\le \min\{q-1,n\}$.
\end{lemma}
\begin{proof}
Since $H_{n,q}$ is a vertex-transitive graph, it suffices to prove that $\phi(H_{n,q},\overrightarrow{0})\le \min\{q-1,n\}$ by Lemma \ref{sym}\ref{a}.

Let $S = \{\overrightarrow{1},\dots,\overrightarrow{q-1}\}$. We will prove that $S$ is a doubly distance resolving set of $H_{n,q}$ on $\overrightarrow{0}$. Let $u$ be a vertex of $H_{n,q}$ and $f(u,c)=\sum_{i=1}^n 1_{u_i=c}$, where $1_{u_i=c}=1$ if $u_i=c$ and $1_{u_i=c}=0$ if $u_i\neq c$. Let $a_{u,c}=d_{H_{n,q}}(u,\overrightarrow{0})-d_{H_{n,q}}(u,\overrightarrow{c})$. Then we have
\begin{equation*}
a_{u,c}=\sum_{i=1}^n 1_{u_i\neq 0}-\sum_{i=1}^n 1_{u_i\neq c}=(n-f(u,0))-(n-f(u,c))=f(u,c)-f(u,0).
\end{equation*}
Since $\sum_{c=0}^{q-1} f(u,c)=n$, we have
\begin{equation*}
\sum_{c=1}^{q-1} a_{u,c} = \sum_{c=1}^{q-1} f(u,c) - (q-1)f(u,0)=\sum_{c=1}^{q-1} f(u,c) - (q-1)\left(n-\sum_{c=1}^{q-1} f(u,c)\right)=q\sum_{c=1}^{q-1} f(u,c)-(q-1)n.
\end{equation*}
Then
\begin{equation*}
d_{H_{n,q}}(u,\overrightarrow{0})=\sum_{c=1}^{q-1}f(u,c)=\frac{\sum_{c=1}^{q-1} a_{u,c}+(q-1)n}{q}.
\end{equation*}
Therefore,
$d_{H_{n,q}}(u,\overrightarrow{0})$ is uniquely determined by the vector $(a_{u,1},\dots,a_{u,q-1})$,
i.e. $S$ is a doubly distance resolving set of $H_{n,q}$ on $\overrightarrow{0}$.

Now we assume that $n\le q-1$. Let $T = \{\overrightarrow{1},\dots,\overrightarrow{n}\}$. Then we will prove that $T$ is a doubly distance resolving set of $H_{n,q}$ on $\overrightarrow{0}$. If there is a pair of vertices $\{u,v\}$ such that $d_{H_{n,q}}(u,\overrightarrow{0})\neq d_{H_{n,q}}(v,\overrightarrow{0})$ and $a_{u,c}=a_{v,c}$ for all $c\in \{1,\dots,n\}$, then, without loss of generality, we assume that $b=d_{H_{n,q}}(v,\overrightarrow{0})-d_{H_{n,q}}(u,\overrightarrow{0})\ge 1$. We have $f(u,0)=n-d_{H_{n,q}}(u,\overrightarrow{0})=n-(d_{H_{n,q}}(v,\overrightarrow{0})-b)=f(v,0)+b$ and $f(u,c)=a_{u,c}+f(u,0)=a_{v,c}+f(v,0)+b=f(v,c)+b$ for all $c\in \{1,\dots,n\}$. But it leads that $\sum_{c=0}^n f(u,c)=\sum_{c=0}^n (f(v,c)+b)\ge n+1$, a contradiction.

\end{proof}
By the Theorem \ref{psile}, we have $\beta(H_{n,q})\le \Psi(H_{n,q})\le \beta(H_{n,q})+\min\{q-1,n\}$. Then we immediately get the following theorem by Theorem \ref{jiang2019}.
\begin{theorem}\label{hamming}
$\Psi(H_{n,q})= (2+o(1))n/\log_q n$ for all $q\ge 2$. \qed
\end{theorem}
Now we focus on the special Hamming graph, that is the hypercube $Q_n$. We have 
\begin{equation*}
    d_{Q_n}(u,x)=\sum_{i=1}^n |u_i-x_i|=\sum_{i=1}^n u_i+x_i-2u_ix_i=u\cdot \overrightarrow{1}+x \cdot \overrightarrow{1}-2(u\cdot x).
\end{equation*}
Note that $u\cdot x=\sum_{i=1}^n u_ix_i$ is the inner product of $u$ and $v$. Then
\begin{equation*}
d_{Q_n}(u,\overrightarrow{0})-d_{Q_n}(u,x)=2(u\cdot x)-x \cdot \overrightarrow{1}.
\end{equation*}

Firstly, we prove the equivalence between the minimal doubly resolving set problem in hypercubes and the coin weighing problem.
\begin{theorem}\label{lin}
For every positive integer $n$, we have $\Psi(Q_n)=M(n)+1$.
\end{theorem}
\begin{proof}
We first prove that $M(n)\le \Psi(Q_n)-1$. By Corollary \ref{e0}, let $S$ be a doubly resolving set of $Q_n$ such that $\overrightarrow{0}\in S$ and $|S|=\Psi(Q_n)$. Now we need to prove that $S'=S\backslash \{\overrightarrow{0}\}$ is a weighing strategy. Suppose not, then there are two distinct vertices $u,v\in V(Q_n)$, such that $u\cdot x=v\cdot x$ for each $x\in S'$. It leads that $d_{Q_n}(u,\overrightarrow{0})-d_{Q_n}(u,x)=2(u\cdot x)-(x\cdot \overrightarrow{1})=2(v\cdot x)-(x\cdot \overrightarrow{1})=d_{Q_n}(v,\overrightarrow{0})-d_{Q_n}(v,x)$ for each $x\in S'$. By Lemma \ref{small}, $S$ is not a doubly resolving set, a contradiction.

Now we prove that $\Psi(Q_n)\le M(n)+1$. Let $S$ be a weighing strategy such that $|S|=M(n)$. Then we prove that $S'=S\cup \{\overrightarrow{0}\}$ is a doubly resolving set. Suppose not, then there are two distinct vertices $u,v\in V(Q_n)$, such that $d_{Q_n}(u,\overrightarrow{0})-d_{Q_n}(u,x)=d_{Q_n}(v,\overrightarrow{0})-d_{Q_n}(v,x)$ for each $x\in S$. Then $u\cdot x = (d_{Q_n}(u,\overrightarrow{0})-d_{Q_n}(u,x)+x\cdot \overrightarrow{1})/2=(d_{Q_n}(v,\overrightarrow{0})-d_{Q_n}(v,x)+x\cdot \overrightarrow{1})/2=v\cdot x$, a contradiction.
\end{proof}

By Theorem \ref{betaPsi}, it is easy to know that
\begin{equation*}
\beta(Q_n)\le \beta(Q_{n+1})=\beta(Q_{n}\square K_2)\le \beta(Q_n)+\Psi(K_2)-1=\beta(Q_n)+1.
\end{equation*}
But it is an open problem whether $\Psi(Q_n)$ has the similar property (see \citet{hertz2020ipbased}). Since we know that the coin weighing problem and minimal doubly resolving set problem in hypercubes are equivalent, it is not difficult to  answer this open problem using the result of the coin weighing problem.
\begin{theorem}\label{nn1}
For every positive integer $n$, we have $\Psi(Q_n)\le \Psi(Q_{n+1})\le \Psi(Q_n)+1$.
\end{theorem}
\begin{proof}
By Theorem \ref{lin}, we need to prove that $M(n)\le M(n+1)\le M(n)+1$.


Let $u=(u_1,\dots,u_n)$ and $v=(v_1,\dots,v_n)$ be two distinct distribution of defective coins. Let $S$ be a weighing strategy for $n+1$ coins such that $|S|=M(n+1)$. Then there exists $x=(x_1,\dots,x_{n+1})\in S$, such that $u'\cdot x\neq v'\cdot x$, where $u' = (u_1,\dots,u_n,0)$ and  $v'=(v_1,\dots,v_n,0)$. Let $x'=(x_1,\dots,x_n)$. Then $u \cdot x'=u' \cdot x\neq  v'\cdot x=v\cdot x'$. Thus, $S'=\{x'=(x_1,\dots,x_n):x=(x_1,\dots,x_{n+1})\in S\}$ is a weighing strategy for $n$ coins, i.e. $M(n)\le M(n+1)$.


Let $u=(u_1,\dots,u_{n+1})$ and $v=(v_1,\dots,v_{n+1})$ be two distinct distribution of defective coins. If $u_{n+1}\neq v_{n+1}$, then $u\cdot y\neq v\cdot y$ where $y=(0,\dots,0,1)$. Now we assume that $u_{n+1}= v_{n+1}$. Let $S$ be a weighing strategy for $n$ coins such that $|S|=M(n)$. Then there exists $x=(x_1,\dots,x_n)\in S$, such that $u'\cdot x\neq v'\cdot x$, where $u' = (u_1,\dots,u_n)$ and  $v'=(v_1,\dots,v_n)$. Let $x'=(x_1,\dots,x_n,0)$. Then $u \cdot x' =u' \cdot x\neq v' \cdot x = v\cdot x'$. Thus, $S'=y\cup \{x'=(x_1,\dots,x_n,0):x=(x_1,\dots,x_n)\in S\}$ is a weighing strategy for $n+1$ coins, i.e. $M(n+1)\le M(n)+1$.
\end{proof}

\section{Folded hypercubes}\label{folded}
It is clear that $f(\fold{u})=\fold{u-x+y}$ is an automorphism of $F_n$ with $f(\fold{x})=\fold{x-x+y}=\fold{y}$. Then $F_n$ is a vertex-transitive graph. In this section, we will use some precise values of $\beta(Q_n)$ and $\beta(F_n)$ that have calculated in \citep{bailey2015metric,caceres2007metric} (see Table \ref{qn9}).
\begin{table}[H]
    \centering
    \begin{tabular}{cccccccccc}
    \hline
    $n$   & 1  & 2 & 3 & 4 & 5 & 6 & 7 & 8  & 9 \\
    $\beta(Q_n)$ & 1 & 2 & 3 & 4 & 4 & 5 & 6 & 6 & 7\\
    $\beta(F_n)$ & - & 1 & 3 & 6 & 4 & 8 & 6 & 11 & -\\
    \hline
    \end{tabular}
    \caption{$\beta(Q_n)$ and $\beta(F_n)$, $n\le 9$}\label{qn9}
\end{table}

\begin{lemma}\label{ge}
For every integer $n\ge 3$, $\beta(F_n)\ge \beta(Q_n)$.
\end{lemma}
\begin{proof}
Let $f:V(F_n)\rightarrow V(Q_n)$ be a function such that $f(\fold{x})=x$ if $x_1=0$ and $f(\fold{x})=\overline{x}$ if $x_1=1$. If $S$ is a vertex set of $F_n$, then $f(S)=\{f(\fold{x}):\fold{x}\in S\}$. It suffices to prove that if $S$ is a resolving set of $F_n$, then $f(S)$ is a resolving set of $Q_n$.

For every two distinct vertices $u,v\in V(Q_n)$ with $u\neq \overline{v}$, since $S$ is a resolving set of $F_n$, there is a vertex $\fold{x}\in S$, such that $d_{F_n}(\fold{u},\fold{x})\neq d_{F_n}(\fold{v},\fold{x})$ with $f(\fold{x})=x$. Let $d_1=d_{F_n}(\fold{u},\fold{x})$ and $d_2=d_{F_n}(\fold{v},\fold{x})$. Then $d_{Q_n}(u,x)\in \{d_1,n-d_1\}$ and $d_{Q_n}(v,x)\in \{d_2,n-d_2\}$. If $d_{Q_n}(u,x)=d_{Q_n}(v,x)$, since $d_1\neq d_2$, we have $d_1+d_2=n$. Since $\max\{d_1,d_2\}\le n/2$, we have $d_1=d_2=n/2$, a contradiction. Therefore, $d_{Q_n}(u,x)\neq d_{Q_n}(v,x)$.

Now we consider the case that $u=\overline{v}$. Since $F_n$ is a vertex-transitive graph, we can assume that $u=\overrightarrow{0},v=\overrightarrow{1}$. Note that for each $y\in V(Q_n)$, we have $d_{Q_n}(u,y)+d_{Q_n}(v,y)=n$. If there is a vertex $x\in f(S)$ such that $d_{Q_n}(u,x)\neq d_{Q_n}(v,x)$, then we have done. Otherwise, for each $x\in f(S)$, we have $d_{Q_n}(u,x)=d_{Q_n}(v,x)=n/2$. It leads that $n$ is even.  Let $S'=\{\fold{x}:\sum_{i=1}^n x_i=n/2\}$. Then $S\subseteq S'$. However, $\{\fold{s},\fold{t}\}$ cannot be resolved by $S'$, where $s=(1,0,0,\dots,0)$ and $t=(0,1,0,\dots,0)$, since $d_{F_n}(\fold{s},\fold{x})=d_{F_n}(\fold{t},\fold{x})=n/2-1$ for each $\fold{x}\in S'$. Then $S'$ is not a resolving set of $F_n$, a contradiction.
\end{proof}

What is more, we prove that equality holds if $n\ge 3$ is odd.
\begin{lemma}\label{oddfold}
If $n\ge 3$ is odd, then $\beta(F_n)=\beta(Q_n)$.
\end{lemma}
\begin{proof}
By Lemma \ref{ge}, it suffices to prove that $\beta(F_n)\le \beta(Q_n)$. Let $g:V(Q_n)\rightarrow V(F_n)$ be a function such that $g(x)=\fold{x}$. If $S$ is a vertex set of $Q_n$, then $g(S)=\{\fold{x}:x\in S\}$.  It suffices to prove that if $S$ is a resolving set of $Q_n$, then $g(S)$ is a resolving set of $F_n$.

For every $x\in S$ and distinct vertices $\fold{u},\fold{v}\in V(F_n)$, if $\fold{x}$ resolves $\fold{u},\fold{v}$, we have done. Otherwise, without loss of generality, we assume that $d_{Q_n}(u,x)=d_{F_n}(\fold{u},\fold{x})=d_{F_n}(\fold{v},\fold{x})=d_{Q_n}(v,x)$. Then $d_{Q_n}(u,x)-d_{Q_n}(v,x)=u\cdot \overrightarrow{1}-v\cdot \overrightarrow{1}-(2u\cdot x-2v\cdot x)=0$. It shows that $u\cdot \overrightarrow{1}+v\cdot \overrightarrow{1}$ is even.

Since $S$ is a resolving set of $Q_n$, there is a vertex $y\in S$ such that $d_{Q_n}(u,y)\neq d_{Q_n}(v,y)$. Let $d_1=d_{Q_n}(u,y)$ and $d_2=d_{Q_n}(v,y)$. Then $d_{F_n}(\fold{u},\fold{y})=\min\{d_1,n-d_1\}$ and $d_{F_n}(\fold{v},\fold{y})=\min\{d_2,n-d_2\}$. If $\fold{y}$ does not resolve $\fold{u}$ and $\fold{v}$, then $d_1+d_2=n$. Besides, $d_1+d_2=2y\cdot \overrightarrow{1}+u\cdot \overrightarrow{1}+v\cdot \overrightarrow{1}-(2u\cdot y+2v\cdot y)$. It leads that $u\cdot \overrightarrow{1}+v\cdot \overrightarrow{1}$ is odd, a contradiction.
\end{proof}

If $n$ is even, equality does not hold in general, such as $\beta(F_4)=6\neq 4=\beta(Q_4)$. If $n$ is even, the following lemma provides the upper bound.
\begin{lemma}\label{evenfold}
For every positive integer $n$, $\beta(F_{n+1})\le 2\beta(Q_n)$.
\end{lemma}
\begin{proof}
For each $x\in V(Q_n)$, let $x^0=(x_1,\dots,x_n,0)$ and $x^1=(x_1,\dots,x_n,1)$ be the two vertices in $V(Q_{n+1})$. It suffices to prove that if $S$ is a resolving set of $Q_n$, then $S'=\{\fold{x^0},\fold{x^1}:x\in S\}$ is a resolving set of $F_{n+1}$.

Let $\fold{u},\fold{v}\in V(F_n)$ be two distinct vertices. Without loss of generality, we assume that $u_{n+1}=v_{n+1}=0$. Let $u'=(u_1,\dots,u_n)$ and $v'=(v_1,\dots,v_n)$. Then since $S$ is a resolving set of $Q_n$, there is a vertex $x\in S$ such that $d_{Q_n}(u',x)\neq d_{Q_n}(v',x)$. Let $d_1=d_{Q_n}(u',x)$ and $d_2=d_{Q_n}(v',x)$. Then 
$d_{F_{n+1}}(\fold{u},\fold{x^0})=\min\{d_1,n+1-d_1\}$, $d_{F_{n+1}}(\fold{u},\fold{x^1})=\min\{d_1+1,n-d_1\}$, $d_{F_{n+1}}(\fold{v},\fold{x^0})=\min\{d_2,n+1-d_2\}$, $d_{F_{n+1}}(\fold{v},\fold{x^1})=\min\{d_2+1,n-d_2\}$. If $\fold{x^0}$ or $\fold{x^1}$ resolves $\{\fold{u},\fold{v}\}$, we have done. Otherwise, we have $d_1+d_2=n+1$ and $(d_1+1)+(d_2+1)=n+1$, a contradiction.
\end{proof}


Recall that $\beta(Q_{n+1})\le \beta(Q_n)+1$ and hence $\beta(Q_{n+m})\le \beta(Q_n)+m$. Since $\beta(Q_5)=4$, $\beta(Q_n)\le \beta(Q_5)+n-5=n-1$ if $n\ge 5$. By Lemma \ref{oddfold},  $\beta(F_n)=\beta(Q_n)\le n-1$ for odd $n\ge 5$. By Lemma \ref{evenfold},  $\beta(F_n)\le 2\beta(Q_{n-1})\le 2n-4$ for even $n\ge 6$. This is a shorter proof of Theorem \ref{zhang2020}.  Since $\beta(Q_9)=7$,  $\beta(Q_n)\le \beta(Q_9)+n-9=n-2$ for odd $n\ge 9$. It implies that Conjecture \ref{zhang2020conj} is false. Furthermore, we have the following asymptotic result of $\beta(F_n)$ by Lemmas \ref{ge}--\ref{evenfold} and $\beta(Q_n)=(2+o(1))n/\log_2 n$.
\begin{theorem}\label{rsfold}
If $n$ is odd, then $\beta(F_n)=(2+o(1))n/\log_2 n$. If $n$ is even, then $(2+o(1))n/\log_2 n\le \beta(F_n)\le (4+o(1))n/\log_2 n$.\qed
\end{theorem}

Now we consider the doubly distance resolving set of $F_n$.
\begin{lemma}\label{ddrsodd}
If $n=2k+1\ge 3$ is odd, then $\phi(F_n)\le (n+1)/2$.
\end{lemma}
\begin{proof}
Let $x^i=(x^i_1,\dots,x^i_n)\in V(Q_n)$ such that
\begin{equation*}
x^i_j= \begin{cases}
1 & j=2i-1 \text{ or }2i\\
0 & \text{otherwise}
\end{cases}
\text{ for }
1\le i \le k
\quad \quad\text{and}\quad \quad
x^{k+1}_j= \begin{cases}
1 & j=2k+1\\
0 & \text{otherwise}.
\end{cases}
\end{equation*}
Let $S = \{\fold{x^1},\dots,\fold{x^{k+1}}\}$. Then it suffices to prove that $S$ is a doubly distance resolving set of $F_n$ on $\foldzero$.

Let $\fold{u}$ be a vertex of $F_n$. Without loss of generality, we assume that $d_{Q_n}(u,\overrightarrow{0}) \le k$. Then $d_{F_n}(\fold{u},\foldzero)=\sum_{j=1}^n u_j$. Let $a_i = d_{F_n}(\fold{u},\foldzero)-d_{F_n}(\fold{u},\fold{x^i})$. Then
\begin{equation*}
    a_i = d_{Q_n}(u,\overrightarrow{0})-\min\{d_{Q_n}(u,x^i),d_{Q_n}(u,\overline{x^i})\} =\max\{2u\cdot x^i-x^i\cdot \overrightarrow{1},2u\cdot \overline{x^i}-\overline{x^i}\cdot \overrightarrow{1}\}.
\end{equation*}
Therefore
\begin{equation*}
a_i=\max\left\{2(u_{2i-1}+u_{2i})-2,2\left(\sum_{j=1}^n u_j - u_{2i}-u_{2i-1}\right)-(n-2)\right\} \text{ for }1\le i \le k
\end{equation*}
and
\begin{equation*}
a_{k+1}=\max\left\{2u_{2k+1}-1,2\left(\sum_{j=1}^n u_j -u_{2k+1}\right)-(n-1)\right\}.
\end{equation*}

Firstly, if $a_i$ is even for every $i\le k$ and $a_{k+1}$ is odd, then $a_i=2(u_{2i-1}+u_{2i})-2$ for every $i\le k$ and $a_{k+1}=2u_{2k+1}-1$. Thus,
\begin{equation*}
    \sum_{j=1}^n u_j=\sum_{i=1}^{k+1} u \cdot x^i=\left(\sum_{i=1}^{k+1} a_i+n\right)/2.
\end{equation*}

Secondly, if $a_{k+1}$ is even, then
\begin{equation*}
    2\left(\sum_{j=1}^n u_j -u_{2k+1}\right)-(n-1)>2u_{2k+1}-1 \Rightarrow 2\sum_{j=1}^n u_j>4u_{2k+1}+n-2\ge n-2=2k-1.
\end{equation*}
Since $d_{Q_n}(u,\overrightarrow{0})=\sum_{j=1}^n u_j\le k$, we have $\sum_{j=1}^n u_j=k$.

Finally, if $a_i$ is odd for some $i\le k$, then $a_i=2\left(\sum_{j=1}^n u_j - u_{2i}-u_{2i-1}\right)-(n-2)$ and
\begin{equation*}
2(u_{2i-1}+u_{2i})-2<2\left(\sum_{j=1}^n u_j - u_{2i}-u_{2i-1}\right)-(n-2)\Rightarrow u_{2i-1}+u_{2i}< \frac{2\sum_{j=1}^n u_j-(n-4)}{4}\le \frac{3}{4}.
\end{equation*}
It leads that $u_{2i-1}+u_{2i}=0$. We have $\sum_{j=1}^n u_j=(a_i+n-2)/2+(u_{2i-1}+u_{2i})=(a_i+n-2)/2$.

From the above discussion, $d_{F_n}(\fold{u},\foldzero)$ is uniquely determined by the vector $(a_1,\dots,a_{k+1})$. Therefore, $S$ is a doubly distance resolving set of $F_n$ on $\foldzero$.
\end{proof}

\begin{lemma}\label{ddrseven}
If $n=2k$ is even, then $\phi(F_n)\le n-1$.
\end{lemma}
\begin{proof}
Let $x^i=(x^i_1,\dots,x^i_n)\in V(Q_n)$ such that for $0\le i \le n$,
\begin{equation*}
x^i_j= \begin{cases}
1 & j\le i\\
0 & \text{otherwise}
\end{cases}
\end{equation*}
Let $S = \{\fold{x^1},\dots,\fold{x^{n-1}}\}$. It suffices to prove that $S$ is a doubly distance resolving set of $F_n$ on $\foldzero$.

Let $\fold{u}$ be a vertex of $F_n$. Without loss of generality, we assume that $d_{Q_n}(u,\overrightarrow{0}) \le k$. Then $d_{F_n}(\fold{u},\foldzero)=\sum_{j=1}^n u_j$. Let $a_i = d_{F_n}(\fold{u},\foldzero)-d_{F_n}(\fold{u},\fold{x^i})$,
\begin{equation*}
    b_i=d_{Q_n}(u,\overrightarrow{0})-d_{Q_n}(u,x^i)=2u\cdot x^i-x^i\cdot \overrightarrow{1}=2\sum_{j=1}^i u_j-i
\end{equation*}
and
\begin{equation*}
    c_i=d_{Q_n}(u,\overrightarrow{0})-d_{Q_n}(u,\overline{x^i})=2u\cdot \overline{x^i}-\overline{x^i}\cdot \overrightarrow{1}=2\sum_{j=i+1}^n u_j-(n-i).
\end{equation*}
Then
\begin{equation*}
a_i=\max\{b_i,c_i\}=\max\left\{2\sum_{j=1}^i u_j-i,2\sum_{j=i+1}^n u_j-(n-i)\right\}.
\end{equation*}

Let $d_i = b_i-c_i$. First, since $b_0=c_n=0$, we have $d_0\le 0$ and $d_n\ge 0$. Second, $|d_{i+1}-d_i| \le |b_{i+1}-b_i|+|c_{i+1}-c_i|=|2u_{i+1}-1|+|1-2u_{i+1}|=2$. Third, since $b_i+c_i=2\sum_{j=1}^n u_j-n$ is even, $d_i$ is even. Combining them with the principle of bisection method, there is a $k$ such that $d_k=0$, i.e. $a_k=b_k=c_k$. Besides, for each $i$ such that $b_i\neq c_i$, we have $a_i = \max\{b_i,c_i\}>(b_i+c_i)/2=\sum_{j=1}^n u_j-n/2=(b_k+c_k)/2=a_k$, i.e. $a_k=\min\{a_i:i\in \{0,1,\dots,n\}\}$. Note that $a_0=a_n=0$. Therefore,
\begin{equation*}
    \sum_{j=1}^n u_j=\sum_{j=1}^k u_j+\sum_{j=k+1}^n u_j=\frac{a_k+k}{2}+\frac{a_k+(n-k)}{2}=a_k+\frac{n}{2}.
\end{equation*}

From the above, $d_{F_n}(\fold{u},\foldzero)$ is uniquely determined by the vector $(a_1,\dots,a_{k+1})$. Therefore, $S$ is a doubly distance resolving set of $F_n$ on $\foldzero$.
\end{proof}
By Theorems \ref{psile} and \ref{rsfold}, we have the following theorem.
\begin{theorem}
If $n\ge 3$ is odd, then $\Psi(F_n)\le \beta(F_n)+(n+1)/2=n/2+o(n)$. If $n$ is even, then  $\Psi(F_n)\le \beta(F_n)+n-1=n+o(n)$.  \qed
\end{theorem}

\section{NP-completeness}\label{npc}
A split graph is a graph whose vertex set is the disjoint union of a clique $C$ and an independent set $I$. In other words, every two vertices in $C$ are connected by an edge, while no two vertices in $I$ are connected by an edge. There is no restriction on edges having one end in $C$ and one end in $I$.

Similarly, a bipartite graph is a graph whose vertex set is the disjoint union of two independent sets $S_1$ and $S_2$. A co-bipartite graph is a graph whose vertex set is the disjoint union of two cliques $C_1$ and $C_2$. 

In this section, we prove that the minimal doubly resolving set problem is NP-hard even for split graphs, bipartite graphs and co-bipartite graphs. The proof is an extension of the proof for metric dimension problem in \citep{epstein2015weighted}. Besides, we omit some details that were mentioned in \citep{epstein2015weighted}.

The $3$-dimensional matching problem is defined as follows. Given three disjoint sets $A,B,C$ such that $|A|=|B|=|C|=n$, and a set of triples $S \subseteq A \times B \times C$, is there a subset $S'\subseteq S$ such that each element of $A \cup B \cup C$ occurs in exactly one of the triples of $S'$. It is well-known that 3-dimensional matching problem is NP-hard, due to \citet{karp1972reducibility}. 

For each subset $S'\subseteq S$, the cost of $S'$ is calculated by $c(S')=|S'|+3n-|\bigcup_{(a,b,c)\in S'}\{a,b,c\}|$. Note that if $S'$ is a 3-dimensional matching, then $c(S')=n$. Let $N=2^{12}n$ and $n'=nN$. Let $\mathcal{A}=\bigcup_{i=1}^N A_i$, $\mathcal{B}=\bigcup_{i=1}^N B_i$ and $\mathcal{C}=\bigcup_{i=1}^N C_i$, where $A_i,B_i,C_i$ are the copies of $A,B,C$ respectively. Let $\mathcal{S}=\bigcup_{i=1}^N S_i$, where $S_i$ is the copy of $S$ corresponding to $A_i,B_i,C_i$. It is clear that $\mathcal{S}'=\bigcup_{i=1}^N S'_i\subseteq \mathcal{S}$ is a 3-dimensional matching if $S'\subseteq S$ is a 3-dimensional matching and $S_i'$ is the copy of $S'$ corresponding to $A_i,B_i,C_i$. Furthermore, \citet{epstein2015weighted} proved the following lemma.
\begin{lemma}[\citet{epstein2015weighted}]\label{gap3dm}
There is a 3-dimensional matching $S'\subseteq S$ if and only if there is a subset $\mathcal{S}'\subseteq \mathcal{S}$ such that $c(\mathcal{S}')\le n'+\sqrt{n'}-1$.
\end{lemma}
Let $\mathcal{S}=\{s_0,s_1,\dots,s_{\tau-1}\}$, $v = \lceil \log_2 \tau \rceil$ and $K=n'+v+5$. Note that $K<n'+\sqrt{n'}-4$. They construct a graph $G$ whose vertices are partitioned into two sets $I=\{s_{\mathcal{A}}, s_{\mathcal{B}}, s_{\mathcal{C}}, s_{\mathcal{D}}\} \cup \mathcal{S}$ and $J=\mathcal{A}\cup \mathcal{B}\cup \mathcal{C}\cup \{d_0,d_1,\dots,d_{v-1}\}$. If $u \in J$ and $v \in I$, then $\{u, v\} \in E$ in the seven following cases (see Figure \ref{tunpc}):
\begin{enumerate}
    \item $u \in \mathcal{A}$ and $v = s_{\mathcal{A}}$.
    \item $u \in \mathcal{B}$ and $v = s_{\mathcal{B}}$.
    \item $u \in \mathcal{C}$ and $v = s_{\mathcal{C}}$.
    \item $u \in \mathcal{A}\cup \mathcal{B}\cup \mathcal{C}$ and $v = s_{\mathcal{D}}$.
    \item $u \in \{a,b,c\}$ and $v = (a,b,c) \in \mathcal{S}$
    \item $u = d_i$ and $v = s_j$ such that $\lfloor j/2^i \rfloor \bmod 2 = 1$
    \item $u = d_i$ and $v = s_{\mathcal{D}}$,
\end{enumerate}

The set of additional edges of $G$ is defined according to the following cases. For the case of bipartite graphs there are no additional edges. For the case of split graphs, $J$ is a clique and $I$ is an independent set, and for the case of co-bipartite graphs both $I$ and $J$ are cliques. Clearly, the construction of the graph $G$ in all cases can be done in polynomial time. Then they prove the following lemma.

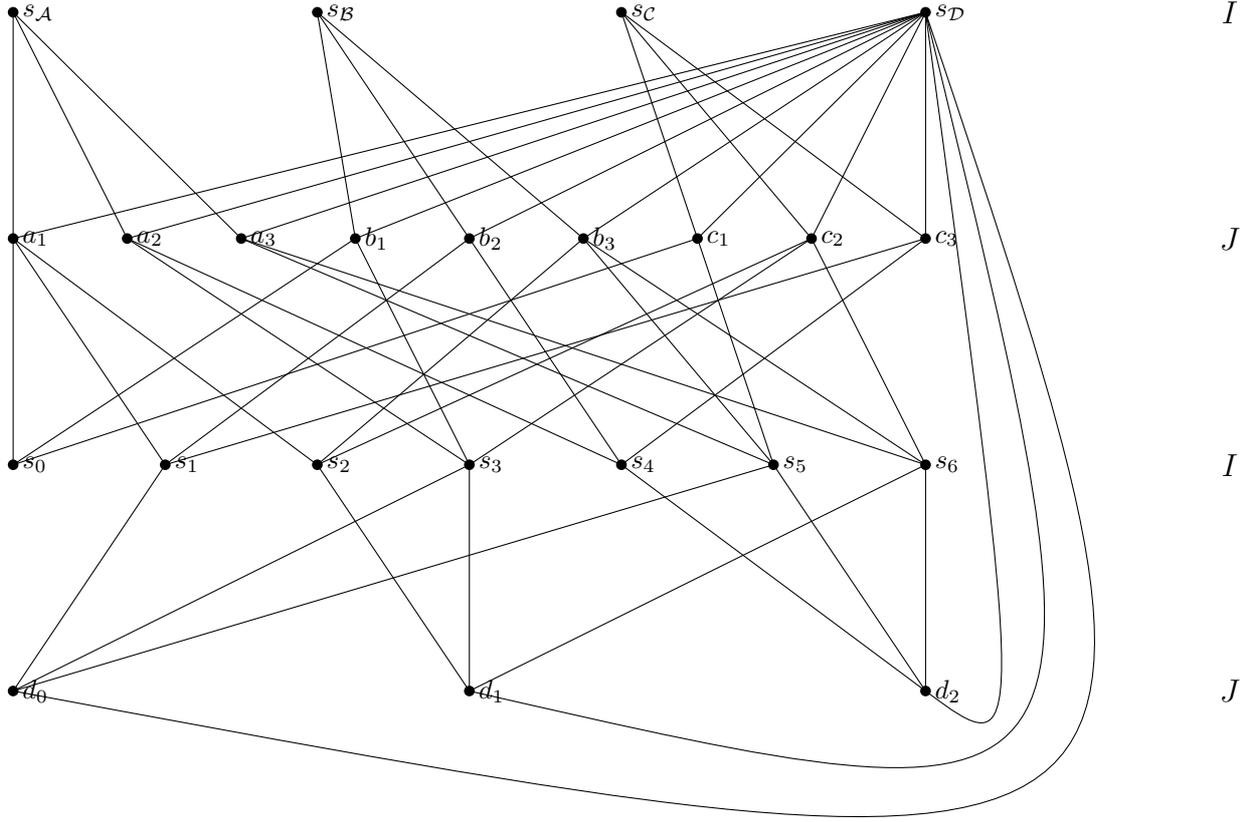
\begin{figure}[htb]
\centering
\begin{tikzpicture}

\coordinate[label=right:{$s_{\mathcal{A}}$}] (sa) at (0,0);
\coordinate[label=right:{$s_{\mathcal{B}}$}] (sb) at (4,0);
\coordinate[label=right:{$s_{\mathcal{C}}$}] (sc) at (8,0);
\coordinate[label=right:{$s_{\mathcal{D}}$}] (sd) at (12,0);
\fill (sa) circle[radius=2pt];
\fill (sb) circle[radius=2pt];
\fill (sc) circle[radius=2pt];
\fill (sd) circle[radius=2pt];
\foreach \i in {1,2,3}
{
\coordinate[label=right:{$a_\i$}] (a\i) at (1.5*\i-1.5,-3);
\fill (a\i) circle[radius=2pt];

\coordinate[label=right:{$b_\i$}] (b\i) at (1.5*\i+3,-3);
\fill (b\i) circle[radius=2pt];

\coordinate[label=right:{$c_\i$}] (c\i) at (1.5*\i+7.5,-3);
\fill (c\i) circle[radius=2pt];

\draw (sa) -- (a\i);
\draw (sb) -- (b\i);
\draw (sc) -- (c\i);
\draw (sd) -- (a\i);
\draw (sd) -- (b\i);
\draw (sd) -- (c\i);
}
\foreach \i in {0,...,6}
{
\coordinate[label=right:{$s_\i$}] (s\i) at (2*\i,-6);
\fill (s\i) circle[radius=2pt];
}
\foreach \i in {0,...,2}
{
\coordinate[label=right:{$d_\i$}] (d\i) at (6*\i,-9);
\fill (d\i) circle[radius=2pt];
}

\draw (sd) .. controls (40/3,-10) .. (d2);
\draw (sd) .. controls (44/3,-11) .. (d1);
\draw (sd) .. controls (16,-12) .. (d0);

\node at (16,0) {\large $I$};
\node at (16,-3) {\large $J$};
\node at (16,-6) {\large $I$};
\node at (16,-9) {\large $J$};

\foreach \i in {0,1,2}
{
\draw (a1) -- (s\i);
}
\foreach \i in {3,4}
{
\draw (a2) -- (s\i);
}
\foreach \i in {5,6}
{
\draw (a3) -- (s\i);
}
\foreach \i in {0,3}
{
\draw (b1) -- (s\i);
}
\foreach \i in {1,4}
{
\draw (b2) -- (s\i);
}
\foreach \i in {2,5,6}
{
\draw (b3) -- (s\i);
}
\foreach \i in {0,5}
{
\draw (c1) -- (s\i);
}
\foreach \i in {2,3,6}
{
\draw (c2) -- (s\i);
}
\foreach \i in {1,4}
{
\draw (c3) -- (s\i);
}

\foreach \i in {1,3,5}
{
\draw (d0) -- (s\i);
}
\foreach \i in {2,3,6}
{
\draw (d1) -- (s\i);
}
\foreach \i in {4,5,6}
{
\draw (d2) -- (s\i);
}

\end{tikzpicture}
\caption{An example for constructing for the case of bipartite graphs (the instance actually is much larger), where $\mathcal{S}=\{(a_1, b_1, c_1)$, $(a_1, b_2, c_3)$, $(a_1, b_3, c_2)$, $(a_2, b_1, c_2)$, $(a_2, b_2, c_3)$, $(a_3, b_3, c_1)$, $(a_3, b_3, c_2)\}$}\label{tunpc}
\end{figure}

\begin{lemma}[\citet{epstein2015weighted}]\label{rsnpc}
\begin{enumerate}[(a)]
\item If $G$ has a resolving set $L$ such that $|L|\le K$, then there is a subset $\mathcal{S}'\subseteq \mathcal{S}$ such that $c(\mathcal{S}')\le K+3<n'+\sqrt{n'}-1$. \label{nprsa}
\item If there is a 3-dimensional matching $\mathcal{S}'\subseteq \mathcal{S}$, then $L=\mathcal{S}'\cup \{s_{\mathcal{A}}, s_{\mathcal{B}}, s_{\mathcal{C}}, s_{\mathcal{D}}\}\cup \{d_0,d_1,\dots,d_{v-1}\}$ is a resolving set of $G$. Note that $|L|=K-1$. \label{nprsb}
\end{enumerate}
\end{lemma}
Now let us consider the minimal doubly resolving set problem.
\begin{lemma}\label{drsnpc}
\begin{enumerate}[(a)]
\item If $G$ has a doubly resolving set $L$ such that $|L|\le K$, then there is a subset $\mathcal{S}'\subseteq \mathcal{S}$ such that $c(\mathcal{S}')\le K+3<n'+\sqrt{n'}-1$. \label{npdrsa}
\item If there is a 3-dimensional matching $\mathcal{S}'\subseteq \mathcal{S}$, then $L=\mathcal{S}'\cup \{s_{\mathcal{A}}, s_{\mathcal{B}}, s_{\mathcal{C}}, s_{\mathcal{D}}\}\cup \{d_0,d_1,\dots,d_{v-1}\}$ is a doubly resolving set of $G$. Note that $|L|=K-1$. \label{npdrsb}
\end{enumerate}
\end{lemma}
\begin{proof}
By Lemma \ref{rsnpc}\ref{nprsa}, we get \ref{npdrsa} immediately since $L$ is also a resolving set. In order to prove \ref{npdrsb}, by Lemmas \ref{lpsile} and \ref{rsnpc}\ref{nprsb}, it suffices to prove that $L'=\{s_{\mathcal{A}}, s_{\mathcal{B}}, s_{\mathcal{C}}\}\cup \{d_0,d_1,\dots,d_{v-1}\}$ is a doubly distance resolving set on $s_{\mathcal{D}}$.

It is easy to check that for each $u\in V(G)$, $d_G(u,s_{\mathcal{D}})\le 2$. In addition, for the case that $G$ is a co-bipartite graph, $d_G(u,s_{\mathcal{D}})\le 1$. Let $\{u,v\}$ be the pair of vertices with $d_G(u,s_{\mathcal{D}})\neq d_G(v,s_{\mathcal{D}})$. Then there exist only three possibilities:

Case 1: $u=s_{\mathcal{D}}$ and $d_G(v,s_{\mathcal{D}}) = 1$. If $v\in \{d_0,d_1,\dots,d_{v-1}\}$, then $\{u,v\}$ is doubly resolved by $\{u,v\}$. If $v\in \mathcal{A}\cup \mathcal{B}\cup \mathcal{C}$, then without loss of generality, we assume that $v\in \mathcal{A}$. Since $d_G(u,u)-d_G(v,u)=0-1<1-1\le d_G(u,s_{\mathcal{A}})-d_G(v,s_{\mathcal{A}})$, $\{u,v\}$ is doubly resolved by $\{u,s_{\mathcal{A}}\}$. The following situations only happen when $G$ is a co-bipartite graph. If $v\in \{s_{\mathcal{A}}, s_{\mathcal{B}}, s_{\mathcal{C}}\}$, then $\{u,v\}$ is doubly resolved by $\{u,v\}$. If $v\in S$, then $d_G(u,u)-d_G(v,u)=0-1<1-1=d_G(u,s_{\mathcal{A}})-d_G(v,s_{\mathcal{A}})$, i.e. $\{u,v\}$ is doubly resolved by $\{u,s_{\mathcal{A}}\}$.

Case 2: $u=s_{\mathcal{D}}$ and $d_G(v,s_{\mathcal{D}}) = 2$. If $v\in \{s_{\mathcal{A}}, s_{\mathcal{B}}, s_{\mathcal{C}}\}$, then $\{u,v\}$ is doubly resolved by $\{u,v\}$. If $v\in S$, then $d_G(u,u)-d_G(v,u)=0-2<2-2=d_G(u,s_{\mathcal{A}})-d_G(v,s_{\mathcal{A}})$, i.e. $\{u,v\}$ is doubly resolved by $\{u,s_{\mathcal{A}}\}$.

Case 3: $d_G(u,s_{\mathcal{D}}) = 1$ and $d_G(v,s_{\mathcal{D}}) = 2$. If $u\in \mathcal{A}\cup \mathcal{B}\cup \mathcal{C}$, then without loss of generality, we assume that $u\in \mathcal{A}$.
If $v\in S$, then $d_G(u,s_{\mathcal{D}})-d_G(v,s_{\mathcal{D}})=1-2<2-2\le d_G(u,s_{\mathcal{B}})-d_G(v,s_{\mathcal{B}})$, i.e. $\{u,v\}$ is doubly resolved by $\{s_{\mathcal{D}},s_{\mathcal{B}}\}$. If $v\in \{s_{\mathcal{A}}, s_{\mathcal{B}}, s_{\mathcal{C}}\}$, then $d_G(u,s_{\mathcal{D}})-d_G(v,s_{\mathcal{D}})=1-2<1-0\le d_G(u,v)-d_G(v,v)$, i.e. $\{u,v\}$ is doubly resolved by $\{s_{\mathcal{D}},v\}$. Now we assume that $u\in \{d_0,d_1,\dots,d_{v-1}\}$. If $v\in S$, then $d_G(u,s_{\mathcal{D}})-d_G(v,s_{\mathcal{D}})=1-2<2-2\le d_G(u,s_{\mathcal{A}})-d_G(v,s_{\mathcal{A}})$, i.e. $\{u,v\}$ is doubly resolved by $\{s_{\mathcal{D}},s_{\mathcal{A}}\}$. If $v\in \{s_{\mathcal{A}}, s_{\mathcal{B}}, s_{\mathcal{C}}\}$, then $d_G(u,s_{\mathcal{D}})-d_G(v,s_{\mathcal{D}})=1-2<2-0\le d_G(u,v)-d_G(v,v)$, i.e. $\{u,v\}$ is doubly resolved by $\{s_{\mathcal{D}},v\}$.

\end{proof}

Note that by Lemmas \ref{gap3dm} and \ref{drsnpc}, there is a 3-dimensional matching $S'\subseteq S$ if and only if $G$ has a doubly resolving set $L$ such that $|L|\le K$. From the above, we get the following theorem.
\begin{theorem}
Given a value $K$ and a graph $G$ that is a split graph, a bipartite graph or a co-bipartite graph, deciding whether $\Psi(G)\le K$ is NP-complete. \qed
\end{theorem}

\section{Algorithms and experimental results}\label{cal}

Lindstr\"{o}m's construction of weighing strategy is a very creative method. For the details and the correctness, the interested reader is referred to Section 2.4 in \citep{aigner1988combinatorial}.

Let $T$ be a finite set and $\mathcal{F}\subseteq 2^T$ a collection of subsets. $\mathcal{F}$ is called a (simplicial) complex if $A\in \mathcal{F}$ and $B\subseteq A$ imply $B\in \mathcal{F}$. Recall that $M(n)$ is the minimum number of weighings for $n$ coins. Lindstr\"{o}m proved the following theorem.
\begin{theorem}[\citet{lindstrom1965combinatorial}]
$M(\sum_{A\in \mathcal{F}}|A|)\le |\mathcal{F}|-1$ for every complex $\mathcal{F}$.
\end{theorem}

For a positive integer $m$, the binary representation can be written to $m=\sum_{i=1}^t 2^{k_i}$. Then let $F_m=\{k_1,k_2,\dots,k_t\}$ with $F_0=\emptyset$. For example, $10=(1010)_2=2^1+2^3$ and $F_{10}=\{1,3\}$. It is easy to know that $\mathcal{F}_m=\{F_0,F_1,\dots,F_{m-1}\}$ is a complex.  Based on  Theorems \ref{lin} and \ref{nn1}, we can construct a doubly resolving set of $Q_n$ with cardinality $|\mathcal{F}|$  for every complex $\mathcal{F}$ and $n\le \sum_{A\in \mathcal{F}}|A|$. Our algorithm for finding upper bounds of $\Psi(Q_n)$ is given in Algorithm \ref{alg}.

\begin{algorithm}
\caption{Finding upper bounds of $\Psi(Q_n)$.}\label{alg}
\KwIn{A positive integer $m$.}
\KwOut{Upper bounds $P(n)$ of $\Psi(Q_n)$ for $n\le \sum_{A\in \mathcal{F}_{m}}|A|$.}
$N := 0$\;
\For{$i:=1$ \KwTo $m-1$}
{
$N_2:=N$\;
$j:=i$\;
\While{$j>0$}
{
\If{$j \bmod 2 = 1$}
{
$N_2 := N_2+1$\;
}
$j := \lfloor j/2 \rfloor$\;
}
\For{$n:=N+1$ \KwTo $N_2$}
{
$P(n):=i+1$\;
}
$N:=N_2$\;

}

\end{algorithm}

We use $\overline{\beta}_n$ and $\overline{\Psi}_n$ to denote the upper bounds of $\beta(Q_n)$ and $\Psi(Q_n)$, respectively. The genetic algorithm (GA), variable neighborhood search (VNS) algorithm and IP-based swapping (IPBS) algorithm were reported in \citep{kratica2009computing}, \citep{mladenovic2012variable} and \citep{hertz2020ipbased}, respectively. Our computing method is Algorithm \ref{alg}. Due to the limitation of the memory space and computing time, the previous results only compute for $n\le 22$. Except for $n=14,16,18$, our upper bounds are same with them (see Table \ref{pn22}). Note that conversely their results actually improved the Lindstr\"{o}m's upper bounds for coin weighing problem in $n=14,16,18$, i.e. $M(14)\le 8$, $M(16)\le 9$ and $M(18)\le 10$. In addition, our upper bound of $\Psi(Q_{28})$ is better than the upper bound of $\beta(Q_{28})$ that founded by IPBS (see Table \ref{pn28}). Recall that $\beta(Q_{n})\le \Psi(Q_{n})$. What is more, when $29 \le n\le 90$, all of our upper bounds of $\Psi(Q_n)$ are not more than the upper bounds of $\beta(Q_n)$ that is calculated by a dynamic programming (DP) procedure in \citep{nikolic2017symmetry} (see Table \ref{pn93}). Besides, some of our upper bounds of $\Psi(Q_n)$ are even better than their upper bounds of $\beta(Q_n)$.

\begin{table}[H]
\begin{minipage}[b]{0.48\textwidth}
\centering

\begin{tabular}{ccccc}
\hline
$n$  & GA & VNS & IPBS & Our  \\ \hline
1  & 2  & -   & -    & 2   \\
2  & 3  & -   & -    & 3   \\
3  & 4  & -   & -    & 4   \\
4  & 4  & -   & -    & 4   \\
5  & 5  & -   & -    & 5   \\
6  & 6  & -   & -    & 6   \\
7  & 6  & -   & -    & 6   \\
8  & 7  & 7   & 7    & 7   \\
9  & 7  & 7   & 7    & 7   \\
10 & 8  & 8   & 8    & 8   \\
11 & 8  & 8   & 8    & 8   \\
12 & 9  & 8   & 8    & 8   \\
13 & 9  & 9   & 9    & 9   \\
14 & 10 & 9   & 9    & 10  \\
15 & 10 & 10  & 10   & 10  \\
16 & 11 & 10  & 10   & 11  \\
17 & 12 & 11  & 11   & 11  \\
18 & -  & -   & 11   & 12  \\
19 & -  & -   & 12   & 12  \\
20 & -  & -   & 12   & 12  \\
21 & -  & -   & 13   & 13  \\
22 & -  & -   & 13   & 13  \\ \hline
\end{tabular}
\caption{$\overline{\Psi}_n$, $n\le 22$}\label{pn22}

\end{minipage}
\begin{minipage}[b]{0.48\textwidth}
\centering
\begin{tabular}{ccc}
\hline
$n$  & $\overline{\beta}_n$ (IPBS) & $\overline{\Psi}_n$ (Our)  \\ \hline
23 & 13                                 & 14                               \\
24 & 14                                 & 14                               \\
25 & 14                                 & 14                               \\
26 & 15                                 & 15                               \\
27 & 15                                 & 15                               \\
28 & 16                                 & 15                               \\ \hline
\end{tabular}
\caption{$\overline{\beta}_n$ and $\overline{\Psi}_n$, $23\le n\le 28$}\label{pn28}
\end{minipage}
\end{table}

\begin{table}[H]
\centering
\begin{tabular}{cccccccccccccc}
\hline
$n$    & 29 & 30 & 31 & 32 & 33 & 34 & 35 & 36 & 37 & 38 & 39 & 40 & 41 \\
$\overline{\beta}_n$ (DP) & 16 & 16 & 16 & 16 & 17 & 18 & 19 & 19 & 20 & 21 & 21 & 22 & 22 \\
$\overline{\Psi}_n$ (Our)  & 16 & 16 & 16 & 16 & 17 & 18 & 18 & 19 & 19 & 20 & 20 & 20 & 21 \\
\hline
$n$    & 42 & 43 & 44 & 45 & 46 & 47 & 48 & 49 & 50 & 51 & 52 & 53 & 54 \\
$\overline{\beta}_n$ (DP) & 23 & 23 & 23 & 24 & 24 & 25 & 25 & 26 & 27 & 28 & 28 & 29 & 30 \\
$\overline{\Psi}_n$ (Our)  & 21 & 22 & 22 & 22 & 23 & 23 & 23 & 24 & 24 & 24 & 24 & 25 & 25 \\
\hline
$n$    & 55 & 56 & 57 & 58 & 59 & 60 & 61 & 62 & 63 & 64 & 65 & 66 & 67 \\
$\overline{\beta}_n$ (DP) & 30 & 30 & 31 & 31 & 32 & 32 & 32 & 32 & 32 & 32 & 32 & 32 & 32 \\
$\overline{\Psi}_n$ (Our)  & 26 & 26 & 26 & 27 & 27 & 27 & 28 & 28 & 28 & 28 & 29 & 29 & 29 \\
\hline
$n$    & 68 & 69 & 70 & 71 & 72 & 73 & 74 & 75 & 76 & 77 & 78 & 79 & 80 \\
$\overline{\beta}_n$ (DP) & 32 & 32 & 32 & 32 & 32 & 32 & 32 & 32 & 32 & 32 & 32 & 32 & 32 \\
$\overline{\Psi}_n$ (Our)  & 30 & 30 & 30 & 30 & 31 & 31 & 31 & 31 & 32 & 32 & 32 & 32 & 32 \\
\hline
$n$    & 81 & 82 & 83 & 84 & 85 & 86 & 87 & 88 & 89 & 90 & 91 & 92 & 93 \\
$\overline{\beta}_n$ (DP) & 33 & 34 & 35 & 35 & 36 & 37 & 37 & 38 & 38 & 39 & -  & -  & -  \\
$\overline{\Psi}_n$ (Our)  & 33 & 34 & 34 & 35 & 35 & 36 & 36 & 36 & 37 & 37 & 38 & 38 & 38\\
\hline
\end{tabular}
\caption{$\overline{\beta}_n$ and $\overline{\Psi}_n$, $29\le n\le 93$}\label{pn93}
\end{table}

Recall that \citet{chartrand2000resolvability} and \citet{kratica2009computing} have given the 0--1 integer linear programming formulations for the metric dimension problem and the minimal doubly resolving set problem respectively. Using the similar method, we give the 0--1 integer linear programming formulations for computing $\phi(G,s)$.

For a doubly distance resolving set $S$ of $G$ on $s$, let
\begin{equation*}
    x_t=\begin{cases}
        1 & \text{if } t\in S\\
        0 & \text{otherwise}.
    \end{cases}
\end{equation*}
Let $T(G,s)=\{(u,v):d_G(u,s)\neq d_G(v,s)\}$. Let
\begin{equation*}
    A_{(u,v),(s,t)}=\begin{cases}
        1 & \text{if } d_G(u,s)-d_G(u,t)\neq d_G(v,s)-d_G(v,t)\\
        0 & \text{otherwise}.
    \end{cases}
\end{equation*}
The following 0--1 linear programming model of calculating the value of $\phi(G,s)$ can be formulated as:
\begin{alignat}{4}
    & \min         & \quad & \sum_{t\in V(G)} x_t  & \quad &\\\label{first}
    & \text{s.t.}  &       & \sum_{t\in V(G)} A_{(u,v),(s,t)}x_t\ge 1  &       &\forall (u,v)\in T(G,s)\\
    &              &       & x_t \in \{0,1\}   &       &\forall t\in V(G). \label{last}
\end{alignat}
By Lemma \ref{small}, it is easy to see that each feasible solution of \eqref{first}-\eqref{last} defines a doubly distance resolving set $S$ of $G$ on $s$, and vice versa.

We use 0--1 linear programming model to compute $\beta(F_n)$, $\phi(F_n)$ and $\Psi(F_n)$ for $n\le 10$ by Gurobi Optimizer (see Table \ref{fn10}). Note that the values of $\beta(F_n)$ for $n\le 8$ have computed in \citep{bailey2015metric}.
\begin{table}[H]
    \centering
    \begin{tabular}{cccccccccc}
        \hline
    $n$     & 2 & 3 & 4 & 5 & 6 & 7 & 8  & 9 & 10 \\
    $\beta(F_n)$ & 1 & 3 & 6 & 4 & 8 & 6 & 11 & 7 & $\le$ 14 \\
    $\Psi(F_n)$ & 2 & 3 & 6 & 5 & 9 & 6 & 11 & 7 & $\le$ 14 \\
    $\phi(F_n)$ & 1 & 1 & 3 & 2 & 5 & 3 & 6  & 3 & $\le$ 8 \\
    \hline
    \end{tabular}
    \caption{$\beta(F_n)$, $\phi(F_n)$ and $\Psi(F_n)$, $n\le 10$}\label{fn10}
\end{table}

\section{Open problems}\label{con}

In Section \ref{cube}, we  proved that $\Psi(H_{n,q})\le \beta(H_{n,q})+q-1$. However, we do not know whether it is best possible. We pose the following question.
\begin{question}
For every positive integer $n$, determine the smallest positive value $f(q)$, such that $\Psi(H_{n,q})\le \beta(H_{n,q})+f(q)$.
\end{question}

By the values of $\beta(F_n)$ and $\Psi(F_n)$ for $n\le 10$ in Table \ref{fn10}, we raise the following conjecture.
\begin{conjecture}
For every integer $n\ge 2$, $\beta(F_n)\le \Psi(F_n) \le \beta(F_n)+1$.
\end{conjecture}

We observe that $\beta(F_{2n})\approx 2\beta(F_{2n-1})$, $\Psi(F_{2n})\approx 2\Psi(F_{2n-1})$ and $\phi(F_{2n})\approx 2\phi(F_{2n-1})$ when $n$ is small by Table \ref{fn10}. Besides, it seems to remain true when $n$ is large by comparing Lemma \ref{oddfold} with Lemma \ref{evenfold}, as well as comparing Lemma \ref{ddrsodd} with Lemma \ref{ddrseven}. We pose the following question and conjecture that the values are 2.

\begin{question}
    Determine the values of 
    \[\lim_{n\rightarrow +\infty} \frac{\beta(F_{2n})}{\beta(F_{2n-1})}, \lim_{n\rightarrow +\infty} \frac{\Psi(F_{2n})}{\Psi(F_{2n-1})}, \lim_{n\rightarrow +\infty}  \frac{\phi(F_{2n})}{\phi(F_{2n-1})}.\]
\end{question}

\bibliographystyle{elsarticle-num-names-alpha}
\bibliography{ref}

\begin{thebibliography}{33}
\providecommand{\natexlab}[1]{#1}
\providecommand{\url}[1]{\texttt{#1}}
\providecommand{\urlprefix}{URL }
\expandafter\ifx\csname urlstyle\endcsname\relax
  \providecommand{\doi}[1]{doi:\discretionary{}{}{}#1}\else
  \providecommand{\doi}[1]{doi:\discretionary{}{}{}\begingroup
  \urlstyle{rm}\url{#1}\endgroup}\fi
\providecommand{\bibinfo}[2]{#2}

\bibitem[{Aigner(1988)}]{aigner1988combinatorial}
\bibinfo{author}{M.~Aigner}, \bibinfo{title}{Combinatorial search},
  Wiley-Teubner Series in Computer Science, \bibinfo{publisher}{John Wiley \&
  Sons, Ltd., Chichester; B. G. Teubner, Stuttgart}, \bibinfo{year}{1988}.

\bibitem[{Bailey(2015)}]{bailey2015metric}
\bibinfo{author}{R.~F. Bailey}, \bibinfo{title}{The metric dimension of small
  distance-regular and strongly regular graphs}, \bibinfo{journal}{Australas.
  J. Combin.} \bibinfo{volume}{62} (\bibinfo{year}{2015})
  \bibinfo{pages}{18--34}.

\bibitem[{Bailey and Cameron(2011)}]{bailey2011base}
\bibinfo{author}{R.~F. Bailey}, \bibinfo{author}{P.~J. Cameron},
  \bibinfo{title}{Base size, metric dimension and other invariants of groups
  and graphs}, \bibinfo{journal}{Bull. Lond. Math. Soc.}
  \bibinfo{volume}{43}~(\bibinfo{number}{2}) (\bibinfo{year}{2011})
  \bibinfo{pages}{209--242}.

\bibitem[{Beerliova et~al.(2005)Beerliova, Eberhard, Erlebach, Hall, Hoffmann,
  Mihal'\'{a}k, and Ram}]{beerliova2005network}
\bibinfo{author}{Z.~Beerliova}, \bibinfo{author}{F.~Eberhard},
  \bibinfo{author}{T.~Erlebach}, \bibinfo{author}{A.~Hall},
  \bibinfo{author}{M.~Hoffmann}, \bibinfo{author}{M.~Mihal'\'{a}k},
  \bibinfo{author}{L.~S. Ram}, \bibinfo{title}{Network discovery and
  verification}, in: \bibinfo{booktitle}{Graph-theoretic concepts in computer
  science}, vol. \bibinfo{volume}{3787} of \emph{\bibinfo{series}{Lecture Notes
  in Comput. Sci.}}, \bibinfo{publisher}{Springer, Berlin},
  \bibinfo{pages}{127--138}, \bibinfo{year}{2005}.

\bibitem[{C\'{a}ceres et~al.(2007)C\'{a}ceres, Hernando, Mora, Pelayo, Puertas,
  Seara, and Wood}]{caceres2007metric}
\bibinfo{author}{J.~C\'{a}ceres}, \bibinfo{author}{C.~Hernando},
  \bibinfo{author}{M.~Mora}, \bibinfo{author}{I.~M. Pelayo},
  \bibinfo{author}{M.~L. Puertas}, \bibinfo{author}{C.~Seara},
  \bibinfo{author}{D.~R. Wood}, \bibinfo{title}{On the metric dimension of
  {C}artesian products of graphs}, \bibinfo{journal}{SIAM J. Discrete Math.}
  \bibinfo{volume}{21}~(\bibinfo{number}{2}) (\bibinfo{year}{2007})
  \bibinfo{pages}{423--441}.

\bibitem[{Cantor and Mills(1966)}]{cantor1966determination}
\bibinfo{author}{D.~G. Cantor}, \bibinfo{author}{W.~H. Mills},
  \bibinfo{title}{Determination of a subset from certain combinatorial
  properties}, \bibinfo{journal}{Canadian J. Math.} \bibinfo{volume}{18}
  (\bibinfo{year}{1966}) \bibinfo{pages}{42--48}.

\bibitem[{Chartrand et~al.(2000)Chartrand, Eroh, Johnson, and
  Oellermann}]{chartrand2000resolvability}
\bibinfo{author}{G.~Chartrand}, \bibinfo{author}{L.~Eroh},
  \bibinfo{author}{M.~A. Johnson}, \bibinfo{author}{O.~R. Oellermann},
  \bibinfo{title}{Resolvability in graphs and the metric dimension of a graph},
  \bibinfo{journal}{Discrete Appl. Math.}
  \bibinfo{volume}{105}~(\bibinfo{number}{1-3}) (\bibinfo{year}{2000})
  \bibinfo{pages}{99--113}.

\bibitem[{Chen and Wang(2014)}]{chen2014approximability}
\bibinfo{author}{X.~Chen}, \bibinfo{author}{C.~Wang},
  \bibinfo{title}{Approximability of the minimum weighted doubly resolving set
  problem}, in: \bibinfo{booktitle}{Computing and combinatorics}, vol.
  \bibinfo{volume}{8591} of \emph{\bibinfo{series}{Lecture Notes in Comput.
  Sci.}}, \bibinfo{publisher}{Springer, Cham}, \bibinfo{pages}{357--368},
  \bibinfo{year}{2014}.

\bibitem[{Chv\'{a}tal(1983)}]{chvatal1983mastermind}
\bibinfo{author}{V.~Chv\'{a}tal}, \bibinfo{title}{Mastermind},
  \bibinfo{journal}{Combinatorica} \bibinfo{volume}{3}~(\bibinfo{number}{3-4})
  (\bibinfo{year}{1983}) \bibinfo{pages}{325--329}.

\bibitem[{Epstein et~al.(2015)Epstein, Levin, and
  Woeginger}]{epstein2015weighted}
\bibinfo{author}{L.~Epstein}, \bibinfo{author}{A.~Levin},
  \bibinfo{author}{G.~J. Woeginger}, \bibinfo{title}{The (weighted) metric
  dimension of graphs: hard and easy cases}, \bibinfo{journal}{Algorithmica}
  \bibinfo{volume}{72}~(\bibinfo{number}{4}) (\bibinfo{year}{2015})
  \bibinfo{pages}{1130--1171}.

\bibitem[{Erd\H{o}s and R\'{e}nyi(1963)}]{erdos1963two}
\bibinfo{author}{P.~Erd\H{o}s}, \bibinfo{author}{A.~R\'{e}nyi},
  \bibinfo{title}{On two problems of information theory},
  \bibinfo{journal}{Magyar Tud. Akad. Mat. Kutat\'{o} Int. K\"{o}zl.}
  \bibinfo{volume}{8} (\bibinfo{year}{1963}) \bibinfo{pages}{229--243}.

\bibitem[{Harary and Melter(1976)}]{harary1976metric}
\bibinfo{author}{F.~Harary}, \bibinfo{author}{R.~A. Melter}, \bibinfo{title}{On
  the metric dimension of a graph}, \bibinfo{journal}{Ars Combin.}
  \bibinfo{volume}{2} (\bibinfo{year}{1976}) \bibinfo{pages}{191--195}.

\bibitem[{Hernando et~al.(2010)Hernando, Mora, Pelayo, Seara, and
  Wood}]{jernando2010extremal}
\bibinfo{author}{C.~Hernando}, \bibinfo{author}{M.~Mora},
  \bibinfo{author}{I.~M. Pelayo}, \bibinfo{author}{C.~Seara},
  \bibinfo{author}{D.~R. Wood}, \bibinfo{title}{Extremal graph theory for
  metric dimension and diameter}, \bibinfo{journal}{Electron. J. Combin.}
  \bibinfo{volume}{17}~(\bibinfo{number}{1}) (\bibinfo{year}{2010})
  \bibinfo{pages}{Research Paper 30, 28}.

\bibitem[{Hertz(2020)}]{hertz2020ipbased}
\bibinfo{author}{A.~Hertz}, \bibinfo{title}{An {IP}-based swapping algorithm
  for the metric dimension and minimal doubly resolving set problems in
  hypercubes}, \bibinfo{journal}{Optim. Lett.}
  \bibinfo{volume}{14}~(\bibinfo{number}{2}) (\bibinfo{year}{2020})
  \bibinfo{pages}{355--367}.

\bibitem[{Jiang and Polyanskii(2019)}]{jiang2019metric}
\bibinfo{author}{Z.~Jiang}, \bibinfo{author}{N.~Polyanskii}, \bibinfo{title}{On
  the metric dimension of {C}artesian powers of a graph}, \bibinfo{journal}{J.
  Combin. Theory Ser. A} \bibinfo{volume}{165} (\bibinfo{year}{2019})
  \bibinfo{pages}{1--14}.

\bibitem[{Kabatianski et~al.(2000)Kabatianski, Lebedev, and
  Thorpe}]{kabatianski2000mastermind}
\bibinfo{author}{G.~Kabatianski}, \bibinfo{author}{V.~Lebedev},
  \bibinfo{author}{J.~Thorpe}, \bibinfo{title}{The {M}astermind game and the
  rigidity of the {H}amming space}, in: \bibinfo{booktitle}{2000 IEEE
  International Symposium on Information Theory (Cat. No. 00CH37060)},
  \bibinfo{organization}{IEEE}, \bibinfo{pages}{375}, \bibinfo{year}{2000}.

\bibitem[{Kabatyanski\u{\i} and Lebedev(2018)}]{kabatiansky2018metric}
\bibinfo{author}{G.~A. Kabatyanski\u{\i}}, \bibinfo{author}{V.~S. Lebedev},
  \bibinfo{title}{On the metric dimension of nonbinary {H}amming spaces},
  \bibinfo{journal}{Problemy Peredachi Informatsii}
  \bibinfo{volume}{54}~(\bibinfo{number}{1}) (\bibinfo{year}{2018})
  \bibinfo{pages}{54--62}.

\bibitem[{Karp(1972)}]{karp1972reducibility}
\bibinfo{author}{R.~M. Karp}, \bibinfo{title}{Reducibility among combinatorial
  problems}, in: \bibinfo{booktitle}{Complexity of computer computations
  ({P}roc. {S}ympos., {IBM} {T}homas {J}. {W}atson {R}es. {C}enter, {Y}orktown
  {H}eights, {N}.{Y}., 1972)}, \bibinfo{pages}{85--103}, \bibinfo{year}{1972}.

\bibitem[{Khuller et~al.(1996)Khuller, Raghavachari, and
  Rosenfeld}]{khuller1996landmarks}
\bibinfo{author}{S.~Khuller}, \bibinfo{author}{B.~Raghavachari},
  \bibinfo{author}{A.~Rosenfeld}, \bibinfo{title}{Landmarks in graphs},
  \bibinfo{journal}{Discrete Appl. Math.}
  \bibinfo{volume}{70}~(\bibinfo{number}{3}) (\bibinfo{year}{1996})
  \bibinfo{pages}{217--229}.

\bibitem[{Knuth(7677)}]{knuth1976computer}
\bibinfo{author}{D.~E. Knuth}, \bibinfo{title}{The computer as master mind},
  \bibinfo{journal}{J. Recreational Math.}
  \bibinfo{volume}{9}~(\bibinfo{number}{1}) (\bibinfo{year}{1976/77})
  \bibinfo{pages}{1--6}.

\bibitem[{Kratica et~al.(2009{\natexlab{b}})Kratica,
  Kova\v{c}evi\'{c}-Vuj\v{c}i\'{c}, and
  \v{C}angalovi\'{c}}]{kratica2009computingB}
\bibinfo{author}{J.~Kratica},
  \bibinfo{author}{V.~Kova\v{c}evi\'{c}-Vuj\v{c}i\'{c}},
  \bibinfo{author}{M.~\v{C}angalovi\'{c}}, \bibinfo{title}{Computing the metric
  dimension of graphs by genetic algorithms}, \bibinfo{journal}{Comput. Optim.
  Appl.} \bibinfo{volume}{44}~(\bibinfo{number}{2})
  (\bibinfo{year}{2009}{\natexlab{b}}) \bibinfo{pages}{343--361}.

\bibitem[{Kratica et~al.(2012{\natexlab{b}})Kratica,
  Kova\v{c}evi\'{c}-Vuj\v{c}i\'{c}, \v{C}angalovi\'{c}, and
  Stojanovi\'{c}}]{kratica2012minimalA}
\bibinfo{author}{J.~Kratica},
  \bibinfo{author}{V.~Kova\v{c}evi\'{c}-Vuj\v{c}i\'{c}},
  \bibinfo{author}{M.~\v{C}angalovi\'{c}}, \bibinfo{author}{M.~Stojanovi\'{c}},
  \bibinfo{title}{Minimal doubly resolving sets and the strong metric dimension
  of {H}amming graphs}, \bibinfo{journal}{Appl. Anal. Discrete Math.}
  \bibinfo{volume}{6}~(\bibinfo{number}{1})
  (\bibinfo{year}{2012}{\natexlab{b}}) \bibinfo{pages}{63--71}.

\bibitem[{Kratica et~al.(2012{\natexlab{a}})Kratica,
  Kova\v{c}evi\'{c}-Vuj\v{c}i\'{c}, \v{C}angalovi\'{c}, and
  Stojanovi\'{c}}]{kratica2012minimalB}
\bibinfo{author}{J.~Kratica},
  \bibinfo{author}{V.~Kova\v{c}evi\'{c}-Vuj\v{c}i\'{c}},
  \bibinfo{author}{M.~\v{C}angalovi\'{c}}, \bibinfo{author}{M.~Stojanovi\'{c}},
  \bibinfo{title}{Minimal doubly resolving sets and the strong metric dimension
  of some convex polytopes}, \bibinfo{journal}{Appl. Math. Comput.}
  \bibinfo{volume}{218}~(\bibinfo{number}{19})
  (\bibinfo{year}{2012}{\natexlab{a}}) \bibinfo{pages}{9790--9801}.

\bibitem[{Kratica et~al.(2009{\natexlab{a}})Kratica, \v{C}angalovi\'{c}, and
  Kova\v{c}evi\'{c}-Vuj\v{c}i\'{c}}]{kratica2009computing}
\bibinfo{author}{J.~Kratica}, \bibinfo{author}{M.~\v{C}angalovi\'{c}},
  \bibinfo{author}{V.~Kova\v{c}evi\'{c}-Vuj\v{c}i\'{c}},
  \bibinfo{title}{Computing minimal doubly resolving sets of graphs},
  \bibinfo{journal}{Comput. Oper. Res.}
  \bibinfo{volume}{36}~(\bibinfo{number}{7})
  (\bibinfo{year}{2009}{\natexlab{a}}) \bibinfo{pages}{2149--2159}.

\bibitem[{Lindstr\"{o}m(1964)}]{lindstrom1964combinatory}
\bibinfo{author}{B.~Lindstr\"{o}m}, \bibinfo{title}{On a combinatory detection
  problem. {I}}, \bibinfo{journal}{Magyar Tud. Akad. Mat. Kutat\'{o} Int.
  K\"{o}zl.} \bibinfo{volume}{9} (\bibinfo{year}{1964})
  \bibinfo{pages}{195--207}.

\bibitem[{Lindstr\"{o}m(1965)}]{lindstrom1965combinatorial}
\bibinfo{author}{B.~Lindstr\"{o}m}, \bibinfo{title}{On a combinatorial problem
  in number theory}, \bibinfo{journal}{Canad. Math. Bull.} \bibinfo{volume}{8}
  (\bibinfo{year}{1965}) \bibinfo{pages}{477--490}.

\bibitem[{Mladenovi\'{c} et~al.(2012)Mladenovi\'{c}, Kratica,
  Kova\v{c}evi\'{c}-Vuj\v{c}i\'{c}, and
  \v{C}angalovi\'{c}}]{mladenovic2012variable}
\bibinfo{author}{N.~Mladenovi\'{c}}, \bibinfo{author}{J.~Kratica},
  \bibinfo{author}{V.~Kova\v{c}evi\'{c}-Vuj\v{c}i\'{c}},
  \bibinfo{author}{M.~\v{C}angalovi\'{c}}, \bibinfo{title}{Variable
  neighborhood search for metric dimension and minimal doubly resolving set
  problems}, \bibinfo{journal}{European J. Oper. Res.}
  \bibinfo{volume}{220}~(\bibinfo{number}{2}) (\bibinfo{year}{2012})
  \bibinfo{pages}{328--337}.

\bibitem[{Nikoli\'{c} et~al.(2017)Nikoli\'{c}, \v{C}angalovi\'{c}, and
  Gruji\v{c}i\'{c}}]{nikolic2017symmetry}
\bibinfo{author}{N.~Nikoli\'{c}}, \bibinfo{author}{M.~\v{C}angalovi\'{c}},
  \bibinfo{author}{I.~Gruji\v{c}i\'{c}}, \bibinfo{title}{Symmetry properties of
  resolving sets and metric bases in hypercubes}, \bibinfo{journal}{Optim.
  Lett.} \bibinfo{volume}{11}~(\bibinfo{number}{6}) (\bibinfo{year}{2017})
  \bibinfo{pages}{1057--1067}.

\bibitem[{Seb\H{o} and Tannier(2004)}]{sebho2004metric}
\bibinfo{author}{A.~Seb\H{o}}, \bibinfo{author}{E.~Tannier}, \bibinfo{title}{On
  metric generators of graphs}, \bibinfo{journal}{Math. Oper. Res.}
  \bibinfo{volume}{29}~(\bibinfo{number}{2}) (\bibinfo{year}{2004})
  \bibinfo{pages}{383--393}.

\bibitem[{Shapiro(1960)}]{shapiro1960problem}
\bibinfo{author}{H.~S. Shapiro}, \bibinfo{title}{Problem {E}1399},
  \bibinfo{journal}{Amer. Math. Monthly}
  \bibinfo{volume}{67}~(\bibinfo{number}{1}) (\bibinfo{year}{1960})
  \bibinfo{pages}{82}.

\bibitem[{Shapiro and Fine(1960)}]{fine1960solution}
\bibinfo{author}{H.~S. Shapiro}, \bibinfo{author}{N.~J. Fine},
  \bibinfo{title}{Elementary {P}roblems and {S}olutions: {S}olutions: {E}1399},
  \bibinfo{journal}{Amer. Math. Monthly}
  \bibinfo{volume}{67}~(\bibinfo{number}{7}) (\bibinfo{year}{1960})
  \bibinfo{pages}{697--698}.

\bibitem[{Slater(1975)}]{slater1975leaves}
\bibinfo{author}{P.~J. Slater}, \bibinfo{title}{Leaves of trees}, in:
  \bibinfo{booktitle}{Proceedings of the {S}ixth {S}outheastern {C}onference on
  {C}ombinatorics, {G}raph {T}heory, and {C}omputing ({F}lorida {A}tlantic
  {U}niv., {B}oca {R}aton, {F}la., 1975)}, \bibinfo{pages}{549--559. Congressus
  Numerantium, No. XIV}, \bibinfo{year}{1975}.

\bibitem[{Zhang et~al.(2020)Zhang, Hou, Hou, Wu, Du, and Gao}]{zhang2020metric}
\bibinfo{author}{Y.~Zhang}, \bibinfo{author}{L.~Hou}, \bibinfo{author}{B.~Hou},
  \bibinfo{author}{W.~Wu}, \bibinfo{author}{D.-Z. Du},
  \bibinfo{author}{S.~Gao}, \bibinfo{title}{On the metric dimension of the
  folded {$n$}-cube}, \bibinfo{journal}{Optim. Lett.}
  \bibinfo{volume}{14}~(\bibinfo{number}{1}) (\bibinfo{year}{2020})
  \bibinfo{pages}{249--257}.

\end{thebibliography}

\end{document}